\documentclass[10pt]{article}
\usepackage[a4paper]{geometry}
\usepackage{amsthm}
\theoremstyle{plain}
\newtheorem{theorem}{Theorem}[section]
\newtheorem{lemma}[theorem]{Lemma}
\theoremstyle{remark}

\usepackage{amsmath,amsfonts,amssymb}
\usepackage{graphicx,verbatim,mathtools}
\usepackage{bm}

\newif \ifPDF
\PDFtrue
\ifPDF \usepackage[breaklinks,bookmarks=false]{hyperref} \fi

\hypersetup{colorlinks, linkcolor=blue, citecolor=blue,
urlcolor=blue, plainpages=false, pdfwindowui=false,
pdfstartview={FitH}}

\renewcommand{\dim}{d}
\newcommand{\abs}[1]{\lvert#1\rvert}
\newcommand{\tends}{\rightarrow}
\newcommand{\norm}[1]{\lVert#1\rVert}
\newcommand{\p}{\partial}

\DeclareMathOperator{\Div}{div}

\newcommand{\R}{\mathbb{R}}
\newcommand{\Om}{\Omega}
\newcommand{\om}{\omega}
\newcommand{\GD}{\Gamma_{\mathrm D}}
\newcommand{\GN}{\Gamma_{\mathrm N}}
\newcommand{\DO}{\partial\Om}
\newcommand{\calP}{\mathcal{P}}
\newcommand{\VK}{\mathcal{V}_K}
\newcommand{\Fa}{\mathcal{F}_\ver}
\newcommand{\calF}{\mathcal{F}}
\newcommand{\calFdag}{\mathcal{F}_{\dagger}}
\newcommand{\calT}{\mathcal{T}}
\newcommand{\calV}{\mathcal{V}}
\newcommand{\calVint}{\mathcal{V}_{\mathrm{int}}}
\newcommand{\calVext}{\mathcal{V}_{\mathrm{ext}}}

\newcommand{\calFext}{\mathcal{F}_{\mathrm{ext}}}
\newcommand{\bL}{\bm{L}^2(\Om)}
\newcommand{\RTN}{\bm{RTN}_p(\calT)}
\newcommand{\RTNc}{\bm{RTN}_{p^\prime}(\calT)}

\newcommand{\RTNK}{\bm{RTN}_p(K)}
\newcommand{\RTNa}{\bm{RTN}_p(\Ta)}
\newcommand{\RTNl}{\bm{RTN}_{p-1}(\calT)}

\newcommand{\Hdiv}{\bm{H}(\Div)}
\newcommand{\Hdivo}{\bm{H}(\Div,\Om)}
\newcommand{\Vh}{V_h}
\newcommand{\bPi}{\bm{\Pi}^{\bm{RTN}}_{hp}}
\newcommand{\bPialt}{\bm{\Pi}_{hp}}
\newcommand{\Pih}{\Pi_{hp}}

\newcommand{\sh}{\bm{\sigma}_h}
\newcommand{\Ta}{\calT^{\ver}}
\newcommand{\oma}{{\om_{\ver}}}
\newcommand{\Hdiva}{\bm{H}(\Div,\oma)}
\newcommand{\psia}{\psi_{\ver}}

\newcommand{\etaosca}{\eta^{\ver}_{\mathrm{osc}}}
\newcommand{\btaua}{\bm{\tau}_h^{\ver}}

\newcommand{\bxi}{\bm{\xi}}
\newcommand{\sha}{\bm{\sigma}_h^{\ver}}
\newcommand{\shad}{\bm{\widetilde{\sigma}}_h^{\ver}}
\newcommand{\rha}{r_h^{\ver}}
\newcommand{\Qa}{Q_h^{\ver}}
\newcommand{\Va}{\bm{V}_h^\ver}
\newcommand{\Vad}{\bm{\widetilde{V}}_{h}^{\ver}}

\newcommand{\Comdag}{C_{P,\dagger}}
\newcommand{\Hdag}{H^1_{\dagger}(\Om)}
\newcommand{\psid}{\psi_{\dagger}}
\newcommand{\Gd}{\Gamma_{N,\dagger}}
\newcommand{\shd}{\bm{\sigma}_h^{\dagger}}
\newcommand{\shc}{\bm{\sigma}_h^{\mathrm{c}}}
\newcommand{\shu}{\bm{\widetilde{\sigma}}_h}

\newcommand{\ver}{{\bm{a}}}
\newcommand{\verbis}{\mathbf{a}}
\newcommand{\bvh}{\bm{v}_h}
\newcommand{\bv}{\bm{v}}
\newcommand{\bn}{\bm{n}}

\title{Discrete $\MakeLowercase{p}$-robust $\Hdiv$-liftings and a posteriori estimates for elliptic problems with $H^{-1}$ source terms\thanks{This
project has received funding from the European Research Council (ERC)
under the European Union's Horizon 2020 research and innovation program
(grant agreement No 647134 GATIPOR).}}
\author{Alexandre~Ern\footnotemark[2] \and Iain~Smears\footnotemark[3] \and Martin~Vohral\'ik\footnotemark[3]}

\begin{document}

\renewcommand{\thefootnote}{\fnsymbol{footnote}}
\footnotetext[2]{Universit\'e Paris-Est, CERMICS (ENPC), 77455 Marne-la-Vall\'{e}e cedex 2, France (alexandre.ern@enpc.fr).}
\footnotetext[3]{INRIA Paris, 2 Rue Simone Iff, 75012 Paris, France (iain.smears@inria.fr, martin.vohralik@inria.fr)}
\renewcommand{\thefootnote}{\arabic{footnote}}

\numberwithin{equation}{section}

\maketitle

\begin{abstract}
We establish the existence of liftings into discrete subspaces of $\Hdiv$ of piecewise polynomial data on locally refined simplicial partitions of polygonal/polyhedral domains. Our liftings are robust with respect to the polynomial degree. This result has important applications in the a posteriori error analysis of parabolic problems, where it permits the removal of so-called transition conditions that link two consecutive meshes. It can also be used in a the posteriori error analysis of elliptic problems, where it allows the treatment of meshes with arbitrary numbers of hanging nodes between elements. We present a constructive proof based on the a posteriori error analysis of an auxiliary elliptic problem with $H^{-1}$ source terms, thereby yielding results of independent interest. In particular, for such problems, we obtain guaranteed upper bounds on the error along with polynomial-degree robust local efficiency of the estimators.
\end{abstract}
\section{Introduction and main results}\label{sec:intro}

We study in this paper two different but connected problems that we introduce separately.

\subsection{Discrete $p$-robust $\Hdiv$-liftings}

First, we are interested in the problem of finding liftings of piecewise polynomial data into piecewise polynomial subspaces of $\Hdiv$, that are robust with respect to the polynomial degree ($p$-robust).
More precisely, let $\Om \subset\R^{\dim}$, $d\in \{2,3\}$, be a bounded Lipschitz polygonal/polyhedral connected open set.
We consider $\bL\coloneqq L^2(\Om;\R^\dim)$, and $\Hdivo \coloneqq \{\bv \in \bL, \, \nabla{\cdot}\bv \in L^2(\Om)\}$.
Consider a partition of the boundary $\Gamma$ of $\Omega$ into two connected components $\GD$ and $\GN$. Let $\calT$ be a given conforming, simplicial, possibly locally refined mesh of $\Om$. We assume that $\calT$ matches $\GD$ and $\GN$ in the sense that every boundary face of the mesh $\calT$ is fully contained either in $\GD$ or in $\GN$.
Let $f$ be a given scalar function and let $\bxi$ be a given vector field, such that $f$
and $\bxi$ are piecewise-polynomials with respect to $\calT$.
We consider the question of finding a piecewise polynomial vector field $\sh$ in the Raviart--Thomas--N\'ed\'elec subspace of $\Hdivo$ over the mesh $\calT$, such that $\nabla{\cdot}\sh = f$ in $\Om$, $\sh{\cdot}\bn=0$ on $\GN$, and such that $\norm{\sh + \bxi}$ is quasi-minimal in the sense of satisfying a bound of the form
\begin{equation} \label{eq:rob_lift}
    \norm{\sh+\bxi} \lesssim \min_{\substack{\bv \in \Hdivo \\ \nabla{\cdot}\bv
    = f \text{ in }\Om \\ \bv{\cdot} \bn = 0\text{ on }\GN}} \norm{\bv + \bxi}.
\end{equation}
The notation $a\lesssim b$ means that $a\leq C b$, with a constant $C$ that can only depend on the shape-regularity of $\calT$ and on the space dimension $\dim$, but is otherwise independent of the domain $\Om$, of the size of the mesh elements in $\calT$, and crucially of the polynomial degree.
Note that the converse bound in~\eqref{eq:rob_lift} holds trivially with constant 1, since $\sh$ is a member of the minimization set considered in the right-hand side.

Problem~\eqref{eq:rob_lift} is also known as the problem of finding a stable right-inverse of the divergence operator, and it plays an important role in numerical analysis.
We remark that many classical approaches based on projection operators that possess commuting diagram properties are not sufficient for $p$-robustness.
Recently, building on~\cite{CostabelMcIntosh2010,DemkowiczGopalakrishnanSchoberl2012}, Braess, Pillwein, and Sch\"oberl~\cite{BraessPillweinSchoberl2009} showed the existence of discrete $p$-robust $\Hdiv$-liftings of piecewise-polynomial data on a patch of triangular elements sharing a common vertex, see~\cite[Thm.~7]{BraessPillweinSchoberl2009}.
The extension to three space dimensions is given in~\cite[Thm.~2.2]{ErnVohralik2016a}. These results imply that equilibrated flux a posteriori error estimators for elliptic problems are $p$-robust. In the present context, the results of~\cite{BraessPillweinSchoberl2009,ErnVohralik2016a} establish the existence of a discrete $p$-robust $\Hdiv$-lifting satisfying~\eqref{eq:rob_lift} when $\calT$ is a set of elements sharing a common vertex, with $\Om$ being the patch composed of these elements.
The existence of $\Hdiv$-liftings on locally refined meshes (and not just on element patches around mesh vertices) has been recently studied in~\cite{AinsworthGuzmanSayas2016}, where the authors considered vanishing interior source terms but nonzero boundary data and studied liftings for a fixed polynomial degree, with constants typically depending on it.

We now present the first main contribution of this work on problem~\eqref{eq:rob_lift}. Let $H^1_*(\Om) \coloneqq H^1_{\GD}(\Om)$ be the subspace of functions in $H^1(\Om)$ with vanishing trace on $\GD$ if $\GD$ is nontrivial, and otherwise (that is, if $\GN=\DO$), let $H^1_*(\Om)\coloneqq H^1(\Om)/\R$ be the space of functions in $H^1(\Om)$ with mean-value zero.
For an integer $p\geq 0$, let $\RTN \subset \bL $ denote the space of piecewise Raviart--Thomas--N\'ed\'elec vector fields of order $p$ with respect to $\calT$; note that in the present notation, we do not impose $\Hdivo$-conformity on $\RTN$. For further details on the notation, see Section~\ref{sec:set} below. Our first main result on discrete $p$-robust $\Hdiv$-liftings is the following theorem.

\begin{theorem}\label{thm:flux_stability} Let $p\geq 1$. For any $f \in \calP_{p-1}(\calT)$ and $\bxi \in \RTNl$, satisfying $(f,1)=0$ if $\GN=\DO$, we have
\begin{equation}\label{eq:stability_rtn}
\min_{\substack{\bvh  \in \Hdivo\cap \RTN \\ \nabla{\cdot}\bvh = f\text{ in }\Om \\ \bvh {\cdot} \bn = 0\text{ on }\GN }} \norm{\bvh + \bxi } \lesssim \min_{\substack{\bv \in \Hdivo \\ \nabla{\cdot}\bv = f \text{ in }\Om \\ \bv{\cdot} \bn = 0 \text{ on }\GN}} \norm{\bv + \bxi}
 =  \max_{v\in H^1_*(\Om)\setminus\{0\}}\frac{ (f,v) - (\bxi , \nabla v)}{\norm{\nabla v}}.
\end{equation}
\end{theorem}

We emphasize that the constant in the stability bound~\eqref{eq:stability_rtn} does not depend on the polynomial degree $p$. The last equality in~\eqref{eq:stability_rtn} follows from classical equivalence results between primal and dual mixed formulations of elliptic problems. 
This identity is actually important in the proof of Theorem~\ref{thm:flux_stability}. The proof, as detailed below, is constructive and consists of the following two steps: first, we pose a primal problem using the data $f$ and $\bxi$ and approximate it using $H^1$-conforming finite elements of degree $p'=1$.
Then we use this approximate solution to build equilibrated fluxes around each vertex of $\calT$ by posing local minimization problems using Raviart--Thomas--N\'ed\'elec spaces of order $p$, and we use the discrete $p$-robust $\Hdiv$-liftings on patches that were described above.

Theorem~\ref{thm:flux_stability} requires that the discrete $\Hdiv$-conforming lifting be of one polynomial degree higher than the piecewise-polynomial data~$f$ and~$\bxi$. At present, we do not know if equal-order $\Hdiv$-liftings retain $p$-robustness for general data.
Nevertheless, the next theorem shows that equal-order $p$-robust $\Hdiv$-liftings are possible when the data take a certain specialised form, and when the norm on the right-hand side of the stability bound is slightly strengthened.
Let $h_\Omega$ stand for the diameter of $\Omega$ and let $\norm{{\cdot}}_{\infty}$ denote the $L^\infty$-norm.

\begin{theorem}\label{thm:flux_stability_2} Let $p\geq 1$. Let
$\psid\in H^1(\Om)\cap \calP_1(\calT)$ be a continuous piecewise affine function with respect to $\calT$. Let $\calFdag$ denote the set of boundary faces $F$ of the mesh $\calT$ such that $\psid|_F = 0$, and let $\Gd = \cup_{F\in\calFdag}F$. Then, for any $f\in \calP_{p-1}(\calT)$ and $\bxi\in\RTNl$, with $(f,\psid) = (\bxi,\nabla \psid)$ if $\Gd = \DO$, we have
\begin{subequations}\begin{align}
    \min_{\substack{\bvh  \in \Hdivo\cap \RTN \\ \nabla{\cdot}\bvh = \psid f- \nabla \psid {\cdot} \bxi \text{ in }\Om \\
    \bvh {\cdot} \bn = 0\text{ on }\Gd }} \norm{\bvh + \psid \bxi}
& \lesssim C(\Omega,\psid) \min_{\substack{\bv  \in \Hdivo \\ \nabla{\cdot}\bv = f \text{ in }\Om}} \norm{\bv + \bxi}  \label{eq:flux_stability_min} \\
& = C(\Omega,\psid)  \max_{v \in H^1_0(\Om)\setminus\{0\}} \frac{(f,v)-(\bxi,\nabla v)}{\norm{\nabla v}},\label{eq:flux_stability_2}
\end{align}\end{subequations}%
where $C(\Omega,\psid) = \norm{\psid}_{\infty} + \Comdag h_{\Omega} \norm{\nabla \psid}_{\infty}$ and $\Comdag$ is the Poincar\'e constant of the space $\Hdag \coloneqq H^1(\Om)/\R$ if $\Gd=\DO$, and $ \Hdag \coloneqq H^1_{\DO\setminus\Gd}(\Om)$ otherwise, i.e.\ $\norm{v} \leq \Comdag h_\Omega \norm{\nabla v}$ for all $v \in \Hdag$.
\end{theorem}

Note that Theorem~\ref{thm:flux_stability_2} is indeed optimal with respect to the polynomial degrees of the data, as $\psid f \in \calP_{p}(\calT)$ and $\psid \bxi \in \RTN$. Note also that the multiplicative factor $\psid$ has been somehow factored out in the right-hand side of~\eqref{eq:flux_stability_min} which, in particular, entails that the infinite-dimensional minimization set does not enforce any prescription on the normal component of the flux at the boundary. The proof of Theorem~\ref{thm:flux_stability_2} is again constructive and uses as a key idea the link to a primal formulation as indicated by the right-hand side of~\eqref{eq:flux_stability_2}.

Theorem~\ref{thm:flux_stability_2} has two important immediate applications.
The first arises in the context of parabolic problems with mesh adaptation between time-steps. Previous a posteriori error analyses (see, e.g.~\cite{Verfurth2003}) required the so-called transition condition, which restricts the extent of mesh-adaptation between time-steps, since the constants in the efficiency of the estimators typically depended on the ratio of the sizes of the elements between the different meshes.
However, in~\cite{ErnSmearsVohralik2016a}, we were able to remove this restriction for the first time, thanks to Theorem~\ref{thm:flux_stability_2} which enables the construction of equilibrated flux a posteriori error estimates with efficiency bounds that do not depend on the mesh adaptation between time-steps.
We refer the reader to \cite{ErnSmearsVohralik2016a} for further details.
A second application of Theorem~\ref{thm:flux_stability_2} arises in the context of nonmatching meshes: following~\cite{DolejsiErnVohralik2016}, Theorem~\ref{thm:flux_stability_2} enables the construction of equilibrated flux a posteriori error estimates without any restriction on the number of levels of hanging nodes. We give the details in Appendix~\ref{sec:appendix}.

\subsection{A posteriori error analysis of problems with $H^{-1}$ source terms}

The study of discrete $\Hdiv$-liftings is connected to the a posteriori
error analysis of elliptic problems with source terms in $H^{-1}(\Om)$, as we
now explain. We start by noting that the right-hand side
of~\eqref{eq:stability_rtn} above corresponds to the $H^1$-seminorm of the
solution of the model problem
\begin{equation}\label{eq:pde}
\begin{aligned}
-\Delta u &= f + \nabla {\cdot} \bxi & & \text{in }\Om,\\
u &= 0 &&\text{on }\GD,\\
\nabla u {\cdot} \bn &= - \bxi{\cdot}\bn &&\text{on }\GN.
\end{aligned}
\end{equation}
We recall that for a general $\bxi \in \bL$, the source term $f+
\nabla{\cdot} \bxi$ and the Neumann boundary condition in \eqref{eq:pde} must
be interpreted in the sense of distributions. The weak formulation
of~\eqref{eq:pde} then looks for $u\in H^1_*(\Om)$ such that
\begin{equation}\label{eq:weak_pde}
(\nabla u,\nabla v) = (f,v) - (\bxi,\nabla v) \qquad \forall\, v\in H^1_*(\Om).
\end{equation}
In this work, we show that the a posteriori error analysis of problem~\eqref{eq:pde} leads to an inherently constructive, practical, and efficient way of computing discrete $p$-robust $\Hdiv$-liftings, thereby justifying Theorems~\ref{thm:flux_stability} and~\ref{thm:flux_stability_2}.
Furthermore, we study the a posteriori error analysis of \eqref{eq:pde} independently, in the most general case $f\in L^2(\Om)$ and $\bxi \in \bL$, because elliptic problems with distributional source terms arise in many other important applications, such as the computation of scalar potentials in the Helmholtz decomposition of vector fields.

In comparison to the literature on a posteriori error estimates for problems with source terms in $L^2(\Om)$, there are comparatively few works treating the case of distributional data. 
Cohen, DeVore, and Nochetto~\cite{CohenDevoreNochetto2012} propose a posteriori error estimates involving the sum of localized negative norms of the source term over the patches of the mesh and weighted norms of the jumps in the finite element solution over the faces of the mesh.
However, it is known from examples~\cite[p.~704]{CohenDevoreNochetto2012} that this estimator can significantly overestimate the error in some cases; this is due to the splitting of the residual into the source term and the jumps of the numerical
solution.
Recently, Kreuzer and Veeser \cite{KreuzerVeeser2016} derived a posteriori error estimates based on low-pass filters that are both reliable and efficient.
Furthermore, the a posteriori error analysis of problems with Dirac delta source terms, which do not lie in $H^{-1}$ if $\dim\geq 2$, is treated, for instance, in~\cite{AgnelliGarauMorin2014,ArayaBehrensRodrigues2006,GaspozMorinVeeser2016}.

The second main contribution of this work is to extend the results of~\cite{BraessPillweinSchoberl2009,ErnVohralik2016a}, where locally efficient and $p$-robust equilibrated flux error estimators for elliptic problems with sources in $L^2(\Om)$ are derived, to problems with source terms in $H^{-1}(\Om)$ such as~\eqref{eq:pde}.
In particular, we prove the following result (detailed notation is given in Section~\ref{sec:set}):
\begin{theorem}\label{thm:aposteriori}
Let $f\in L^2(\Om)$ and $\bxi \in \bL$, with $(f,1)=0$ if $\GN=\DO$, and let
$u\in H^1_*(\Om)$ be the weak solution of problem~\eqref{eq:pde} defined
by~\eqref{eq:weak_pde}. Let $p$ and $p^\prime$ be positive integers with $1
\leq p^\prime \leq p $, and let $u_h\in \Vh\coloneqq H^1_*(\Om)\cap
\calP_{p^\prime}(\calT)$ be the finite element approximation of $u$ such that
\begin{equation}\label{eq:num_scheme}
(\nabla u_h,\nabla v_h) = (f,v_h)-(\bxi,\nabla v_h)  \qquad \forall\, v_h\in \Vh.
\end{equation}
Let the equilibrated flux reconstruction $\sh \in \Hdivo\cap \RTN$ be defined by~\eqref{eq:sha_def} and~\eqref{eq:flux} below.
Then, we have the guaranteed upper bound on the error
\begin{equation}\label{eq:upper_bound}
\norm{\nabla(u-u_h)}^2 \leq \sum_{K\in\calT} \big[\norm{\sh + \bxi + \nabla u_h }_K + \tfrac{h_K}{\pi}\norm{f-\Pih f}_K\big]^2.
\end{equation}
Furthermore, for each $K\in\calT$, we have the local efficiency bound
\begin{equation}\label{eq:local_efficiency}
\norm{\sh + \bxi + \nabla u_h}_K \lesssim \sum_{\ver\in\VK} \big[ \norm{\nabla(u-u_h)}_{\oma} + \etaosca \big],
\end{equation}
where the local data oscillation $\etaosca$ defined by~\eqref{eq:osc} below.
Finally, the global efficiency can be summarized as
\begin{equation}\label{eq:global_efficiency}
\norm{\sh + \bxi + \nabla u_h} \lesssim \norm{\nabla(u-u_h)} + \left\{ \sum_{\ver\in\calV} [\etaosca]^2 \right\}^{\frac{1}{2}}.
\end{equation}
\end{theorem}

We consider in Theorem~\ref{thm:aposteriori} the polynomial degrees $1 \leq p^\prime \leq p$; in the a posteriori analysis of the model problem~\eqref{eq:weak_pde}, one is typically interested in the situation where $p^\prime=p$. However, in the context of discrete $p$-robust $\Hdiv$-liftings, the approximation~$u_h$ only serves as a tool in the analysis, and the choice $p^\prime=1$ turns out to be sufficient for our purposes.

The rest of this paper is organized as follows. We detail the notation in Section~\ref{sec:set}. Section~\ref{sec:equilibrated_flux} then presents the equilibrated flux reconstruction in the setting of $H^{-1}$ source terms.
Sections~\ref{sec:aposteriori_thm_proof}, \ref{sec:flux_thm_proof}, and~\ref{sec:flux_stability_2} respectively prove Theorems~\ref{thm:aposteriori}, \ref{thm:flux_stability}, and~\ref{thm:flux_stability_2}, and Appendix~\ref{sec:appendix} illustrates an application of our results to a posteriori error estimation on meshes with an arbitrary number of levels of hanging nodes.

\section{Setting}\label{sec:set}

We summarize here briefly the notation used in this paper. For an arbitrary open subset $\omega\subset \Om$, we use $({\cdot},{\cdot})_{\om}$ to denote the $L^2$-inner product for scalar- or vector-valued functions on $\omega$, with associated norm $\norm{{\cdot}}_{\om}$. In the special case where $\om = \Om$, we drop the subscript notation, i.e. $\norm{{\cdot}}\coloneqq\norm{{\cdot}}_{\Om}$.
For each mesh element $K\in \calT$ and for a fixed integer $p\geq 1$, let $\calP_p(K)$ denote the space of polynomials of total degree at most $p$ on $K$.
Let $\calP_p(\calT) \subset L^2(\Om)$ denote the space of scalar piecewise-polynomials of degree at most $p$ over $\calT$ and let $\RTN\subset \bL$ denote the piecewise Raviart--Thomas--N\'ed\'elec space, defined by $\RTN \coloneqq \{ \bvh \in \bL,\;\bvh|_K \in \RTNK\}$, where $\RTNK \coloneqq \calP_p(K;\R^\dim) + \calP_p(K)\bm{x}$. Let $\bPi$ denote the vector $\bL$-orthogonal projection operator from $\bL$ onto $\RTN$.
Let $\Pih \colon L^2(\Om) \tends \calP_p(\calT)$ denote the scalar $L^2$-orthogonal projection operator from $L^2(\Om)$ onto $\calP_p(\calT)$.
Finally, let $\bPialt$ denote the vector $\bL$-orthogonal projection operator from $\bL$ onto $\calP_p(\calT;\R^d)$; note that $\bPialt$ is simply obtained by application of $\Pih$ component-wise. We also emphasize that all these projections are elementwise, in particular since $\Hdivo$-conformity is not imposed on the space $\RTN$.

Let $\calF$ denote the set of faces of the mesh, with $\calFext$ denoting the set of all boundary faces of the mesh. For each element~$K\in\calT$, $h_K$ stands for the diameter of $K$. Let $\calV$ denote the set of vertices of the mesh $\calT$. For each $\ver \in \calV$, the function $\psia$ is the hat function associated with $\ver$, and the set~$\oma$ is the interior of the support of $\psia$, with associated diameter $h_{\oma}$.
Furthermore, let $\Ta$ denote the restriction of the mesh $\calT$ to~$\oma$.
In the case where $\GD$ is nontrivial, a vertex $\ver\in \calV$ is said to belong to $\calVint$, the set of interior and Neumann boundary vertices, if $\ver \in \Omega \cup ( \DO\setminus \overline{\GD})$.
Otherwise $\ver$ belongs to $\calVext$, the set of Dirichlet boundary vertices. In the case where $\GD=\emptyset$ and $\GN=\DO$, all vertices are considered to be interior vertices and $\calVint \coloneqq \calV$. Finally, for each element $K\in \calT$, we collect in $\VK$ the set of vertices of $\calV$ belonging in $K$.

\section{Flux equilibration for elliptic problems with $H^{-1}$ source terms}\label{sec:equilibrated_flux}

The construction of the flux equilibration is based on independent local mixed finite element approximations of residual problems over the patches of elements around mesh vertices, in generalization of~\cite{Ainswort2010,BraessPillweinSchoberl2009,BraessSchoberl2008,DestuynderMetivet1999,ErnVohralik2015,ErnVohralik2016a}.
For each $\ver\in\calV$, let $\calP_p(\Ta)$, respectively $\RTNa$, be the restriction of $\calP_p(\calT)$, respectively $\RTN$, to the patch $\Ta$ around the vertex $\ver \in \calV$. The local spatial mixed finite element spaces $\Va$ and $\Qa$ are defined by
\[
\begin{aligned}
\Va & \coloneqq
\begin{cases}
   \left\{\bvh \in \Hdiva \cap \RTNa,\quad \bvh{\cdot} \bn =0\text{ on }\p\oma \right\} & \, \text{if }\ver\in\calVint,\\
    \left\{\bvh \in \Hdiva\cap \RTNa ,\quad \bvh{\cdot} \bn =0\text{ on }\p\oma\setminus\GD  \right\}& \, \text{if }\ver\in\calVext,
\end{cases}
 \\
\Qa & \coloneqq
\begin{cases}
     \left\{ q_h\in \calP_p(\Ta),\quad (q_h,1)_\oma = 0\right\} & \hspace{4.3cm} \text{if }\ver\in\calVint,\\
    \calP_p(\Ta) & \hspace{4.3cm}\text{if }\ver\in\calVext.
\end{cases}
\end{aligned}
\]
For each $\ver\in \calV$, let $\sha \in \Va$ be defined by
\begin{equation}\label{eq:sha_def}
\sha \coloneqq \arg \min_{\substack{\bvh \in \Va \\ \nabla{\cdot}\bvh  = g_h^{\ver} }} \norm{\bvh + \psia( \bxi + \nabla u_h)}_{\oma},
\end{equation}
where
\begin{equation} \label{eq:ga}
    g_h^{\ver}\coloneqq \Pih \big(\psia f - \nabla \psia {\cdot} (\bxi +
    \nabla u_h )\big)|_{\oma} = \left(\Pih(\psia f)- \nabla \psia {\cdot}\left( \bPialt \bxi +    \nabla u_h\right)\right)|_{\oma},
\end{equation}
where the last equality follows from $\Pih(\nabla \psia{\cdot} \bxi)= \nabla \psia {\cdot} \bPialt \bxi$, since $\nabla \psia$ is piecewise constant, and since $u_h$ is a piecewise polynomial of degree at most $p^\prime\leq p$.
It is important to note that $g_h^{\ver}$ satisfies the Neumann compatibility condition~$(g_h^{\ver},1)_\oma =0$ for all $\ver\in\calVint$, i.e.~$g_h^{\ver}\in \Qa$, thereby guaranteeing that $\sha$ from~\eqref{eq:sha_def} is well-defined.
Indeed, this is found by choosing the test function $v_h = \psia$ in~\eqref{eq:num_scheme} when $\GD \neq \emptyset$, as here $\psia \in \Vh \subset H^1_*(\Om)$.
When $\GD = \emptyset$, $\psia \notin \Vh \subset H^1_*(\Om)$ due to the mean-value zero condition on $H^1_*(\Om)$, but the compatibility condition $(f,1) =0$ implies that~\eqref{eq:num_scheme} also holds for $\psia$.

It is well-known that the Euler--Lagrange conditions
for~\eqref{eq:sha_def} are: find $\sha \in \Va$ and $\rha \in \Qa$ (the
Lagrange multiplier of the divergence constraint) such that
\begin{subequations}\label{eq:flux_system}
\begin{align}
&(\sha,\bvh)_\oma - (\nabla{\cdot} \bvh, \rha)_\oma  = - (\psia (\bxi + \nabla u_h) ,\bvh)_\oma   & & \forall\, \bvh \in \Va, \label{eq:flux_system_1}\\
&(\nabla{\cdot} \sha,q_{h})_\oma  =  (\psia f - \nabla \psia {\cdot} ( \bxi +\nabla  u_h),q_{h})_\oma  & & \forall\, q_{h}\in\Qa.\label{eq:local_equilibrium}
\end{align}
\end{subequations}
After extending each $\sha$ by zero in $\Om\setminus \oma$, we define the
equilibrated flux reconstruction $\sh \in \RTN$ by
\begin{equation}\label{eq:flux}
\sh \coloneqq \sum_{\ver\in\calV}\sha.
\end{equation}

\begin{lemma}\label{lem:fl_rec}
Let the equilibrated flux reconstruction $\sh \in \RTN$ be defined
by~\eqref{eq:flux}. Then, $\sh$ belongs to $\Hdivo$ and satisfies
\begin{subequations}\label{eq:equilibration}
\begin{alignat}{2}
\nabla{\cdot} \sh & = \Pih f \qquad & & \text{in }\Om, \label{eq:equilibration1}\\
\sh {\cdot} \bn & = 0 & &\text{on }\GN. \label{eq:equilibration2}
\end{alignat}
\end{subequations}
\end{lemma}
\begin{proof}
The proof follows closely the arguments in~\cite{BraessSchoberl2008,DestuynderMetivet1999,DolejsiErnVohralik2016,ErnVohralik2015}; we sketch it here for the sake of completeness.
First, for any $\ver\in\calV$, the zero extension of $\sha$ belongs to $\Hdivo$ as a result of the boundary conditions in $\Va$. Thus $\sh \in \Hdivo$.
Consider now any element $K \in \calT$ having a face $F$ contained in $\GN$. Then, for each vertex $\ver \in \VK$, the definition of $\Va$ requires that $\sha{\cdot} \bn =0 $ on $F$, thus implying that $\sh{\cdot} \bn =0 $ on $F$.
Since $\overline{\GN} = \cup_{F\subset\GN}\overline{F}$ by hypothesis, we deduce~\eqref{eq:equilibration2}.
Finally, to show~\eqref{eq:equilibration1}, we employ~\eqref{eq:sha_def} and~\eqref{eq:ga}: thus, on each $K \in \Ta$, $\nabla {\cdot} \sh|_K = \sum_{\ver\in\VK} \nabla {\cdot} \sha|_K = \sum_{\ver\in\VK} \big[\Pih \big(\psia f - \nabla \psia {\cdot} (\bxi + \nabla u_h )\big)\big]|_K = \Pih f|_K$, since the hat functions $\{\psia\}_{\ver\in \VK} $ form a partition of unity over $K$.

\end{proof}

For a given vertex $\ver \in \calV$, let the space~$H^1_*(\oma)$ be defined
by
\[
H^1_*(\oma) \coloneqq \begin{cases}
 \{v\in H^1(\oma),\quad (v,1)_{\oma} = 0\} &\text{if }\ver \in \calVint, \\
 \{v\in H^1(\oma),\quad v|_{\p \oma\cap \GD} =0\} &\text{if }\ver\in\calVext.
\end{cases}
\]
We also set $\btaua \coloneqq \bPi(\psia \bxi) + \psia \nabla u_h$.
Then we have the following crucial stability result.
\begin{lemma}\label{lem:main_stability_bound}
For each $\ver\in\calV$, let $\sha$ be defined by~\eqref{eq:sha_def}. Then
\begin{equation}\label{eq:main_stability_bound}
\norm{\sha + \btaua }_{\oma}
\lesssim \min_{\substack{\bm{\sigma} \in \bm{V}^{\ver} \\ \nabla{\cdot}\bm{\sigma} = g_h^{\ver} }} \norm{\bm{\sigma} + \btaua} = \max_{v\in H^1_*(\oma)\setminus\{0\}}
 \frac{ (g_h^{\ver},v)_{\oma} - (\btaua,\nabla v)_{\oma}}{\norm{\nabla v}_{\oma}},
\end{equation}
where $\bm{V}^{\ver}$ denotes the set of all vector fields $\bm{\sigma}\in H(\Div,\oma)$ such that $\bm{\sigma}{\cdot}\bn =0$ either on $\p\oma$ if $\ver\in\calVint$, or on $\p\oma\setminus \GD$ if $\ver\in\calVext$.
\end{lemma}
\begin{proof} 
It follows from the definitions of the projectors $\bPi$ that
$(\btaua,\bvh)_{\oma} = (\psia (\bxi + \nabla u_h),\bvh)_{\oma}$ for all
$\bvh\in \Va$. Therefore,~\eqref{eq:flux_system_1} implies that
$(\sha,\bvh)_{\oma} - (\nabla{\cdot}\bvh,\rha)_{\oma} =-(\btaua,\bvh)_{\oma}$
for all $\bvh\in \Va$. We deduce that~\eqref{eq:sha_def} is equivalent to
\[
\sha = \arg \min_{\substack{\bvh \in \Va \\ \nabla {\cdot} \bvh = g_h^{\ver}}} \norm{\bvh + \btaua}_{\oma}.
\]
Since $g_h^\ver \in \calP_p(\Ta)$ and since $\btaua \in \RTNa$, we
obtain~\eqref{eq:main_stability_bound}
from~\cite[Thm.~7]{BraessPillweinSchoberl2009} in the case of two space
dimensions (up to straightforward adaptations for boundary vertices),
and~\cite[Thm.~2.2]{ErnVohralik2016a} in the case of three space
dimensions.\end{proof}

\section{Proof of Theorem~\ref{thm:aposteriori}}\label{sec:aposteriori_thm_proof}
The proof follows essentially the arguments in~\cite{BraessSchoberl2008,ErnVohralik2015}.

\subsection{Proof of the guaranteed upper bound~\eqref{eq:upper_bound}}\label{sec:upper}
It is straightforward to see from the fact that $u_h \in H^1_*(\Om)$ and from~\eqref{eq:weak_pde} that
\[
    \norm{\nabla(u-u_h)}  = \max_{v\in H^1_*(\Om)\setminus\{0\}} \frac{(f,v)-(\bxi,\nabla v)-(\nabla u_h,\nabla v)}{\norm{\nabla v}}.
\]
Consider $v\in H^1_*(\Om)$ such that $\norm{\nabla v}=1$. Then, by addition and subtraction and the facts that $\nabla{\cdot}\sh = \Pih f$ and that $\sh{\cdot}\bn=0$ on $\GN$ by Lemma~\ref{lem:fl_rec}, we have
\begin{equation}\label{eq:upper_bound_1}
(f,v)-(\bxi,\nabla v)-(\nabla u_h,\nabla v) = (f-\Pih f,v) - (\sh  +\bxi + \nabla u_h,\nabla v).
\end{equation}
Next, we note that since $f-\Pih f$ has mean-value zero over each $K\in \calT$, we obtain the bound $\abs{(f-\Pih f,v)_K }\leq \tfrac{h_K}{\pi}\norm{f-\Pih f}_K \norm{\nabla v}_K$ for each $K\in \calT$ by the Poincar\'e inequality. The upper bound~\eqref{eq:upper_bound} then follows from~\eqref{eq:upper_bound_1} and the Cauchy--Schwarz inequality.

\subsection{Proof of the local efficiency~\eqref{eq:local_efficiency}}\label{sec:proof_loc_eff}

Let us define the data oscillation as
\begin{equation}\label{eq:osc}
 [\etaosca]^2 \coloneqq \sum_{K\in \Ta} \left\{ \frac{h_{K}^2}{p^2}\norm{\psia f - \Pih(\psia f)}_{K}^2+
 \norm{\bxi - \bPialt\bxi}_{K}^2   + \norm{\psia \bxi - \bPi(\psia \bxi)}_{K}^2 \right\}.
\end{equation}
Recall the definition $\btaua \coloneqq \bPi(\psia \bxi) + \psia \nabla u_h$
and note that, for each $K\in \calT$,
\[
\begin{split}
\norm{\sh + \bxi + \nabla u_h}_K
& \leq \sum_{\ver\in\VK}\norm{\sha + \psia ( \bxi +  \nabla u_h )}_K\\
& \leq \sum_{\ver\in\VK}[\norm{\sha + \btaua }_K + \norm{\psia \bxi - \bPi(\psia \bxi)}_K]\\
& \leq \sum_{\ver\in\VK}[\norm{\sha + \btaua }_\oma + \etaosca].
\end{split}
\]
Employing Lemma~\ref{lem:main_stability_bound}, it is enough to bound the right-hand side \eqref{eq:main_stability_bound}. Fix $\ver \in \VK$ and consider $v\in H^1_*(\oma)$ such that $\norm{\nabla v}_{\oma} =1$. We then write $(g_h^{\ver},v)_{\oma} - (\btaua,\nabla v)_{\oma} = \sum_{i=1}^4
E_i$, where
\begin{align*}
E_1 &\coloneqq (f, \psia v)_{\oma} - (\bxi + \nabla u_h, \nabla(\psia v))_{\oma}, &&&
E_2 &\coloneqq (\Pih(\psia f)-\psia f,v)_{\oma},\\
E_3 &\coloneqq ( [ \bxi - \bPialt \bxi ] {\cdot} \nabla \psia ,v)_{\oma}, &&&
E_4 & \coloneqq (\psia \bxi - \bPi(\psia \bxi),\nabla v)_{\oma}.
\end{align*}
Extend $\psia v$ by zero outside of $\oma$; then two situations may arise.
Either $\GN = \DO$, where we use the fact that the weak formulation~\eqref{eq:weak_pde} holds for all test functions in $H^1(\Om)$ by the compatibility $(f,1)_\oma=0$.
Or $\GD \neq \emptyset$, where $\psia v \in H^1_*(\Om)=H^1_{\GD}(\Om)$ for any $v \in H^1_{*}(\oma)$. In both cases, we conclude from~\eqref{eq:weak_pde} that $(\nabla u,\nabla (\psia
v))_{\oma} = (f, \psia v)_{\oma} - (\bxi,\nabla(\psia v))_{\oma}$.
Therefore, we see that $E_1=(\nabla(u-u_h),\nabla(\psia v))_{\oma}$ and thus $\abs{E_1}\leq \norm{\nabla(u-u_h)}_{\oma}\norm{\nabla(\psia v)}_{\oma}$.
Next, we recall that there is a constant depending only on the mesh shape-regularity such that $\norm{\nabla(\psia v)}_{\oma} \lesssim\norm{\nabla v}_{\oma}$ for all $v\in H^1_*(\oma)$ for every $\ver \in \calV$, see~\cite{BraessPillweinSchoberl2009,ErnVohralik2015}.
Hence, using the hypothesis that $\norm{\nabla v}_{\oma}=1$, we find that
$\abs{E_1}\lesssim \norm{\nabla(u-u_h)}_{\oma}$.
Next, we have $E_2 = (\Pih(\psia f)-\psia f, v- \Pih v)$ by orthogonality of the $L^2$-projection, and thus we find that $\abs{E_2}^2\lesssim \sum_{K\in\Ta} \tfrac{h_K^2}{p^2} \norm{\psia f- \Pih(\psia f)}_{K}^2$ by using the approximation bound $\norm{v-\Pih v}_K \lesssim \tfrac{h_K}{p}\norm{\nabla v}_K$ for all $v\in H^1(K)$. Finally, we find that $\abs{E_3} \lesssim \norm{\bxi - \bPialt \bxi}_{\oma}$, and that $\abs{E_4}\leq \norm{\psia\bxi-\bPi(\psia \bxi)}_{\oma}$.
This completes the proof of local efficiency~\eqref{eq:local_efficiency}.

\subsection{Proof of the global efficiency~\eqref{eq:global_efficiency}}
For each $K\in\calT$, we have $(\sh+\bxi + \nabla u_h)|_K = \sum_{\ver\in\VK} [\sha + \psia(\bxi+\nabla u_h)]|_K$ by the partition of unity.
Noting that each element $K$ has $(\dim+1)$ vertices collected in the set $\VK$ since $\calT$ is a simplicial mesh, the Cauchy--Schwarz inequality implies that
\[
\begin{split}
\norm{\sh+\bxi+\nabla u_h}^2 & = \sum_{K\in\calT}\norm{\sh+\bxi+\nabla u_h}^2_K \\
& \leq \sum_{K\in\calT} (\dim +1)\sum_{\ver\in\VK}\norm{\sha+\psia (\bxi+\nabla u_h)}_K^2 \\
& = (\dim+1) \sum_{\ver\in\calV} \norm{\sha+\psia (\bxi+\nabla u_h)}_{\oma}^2.
\end{split}
\]
It follows from the arguments of Section~\ref{sec:proof_loc_eff} that $\norm{\sha+\psia (\bxi+\nabla u_h)}_{\oma} \lesssim \norm{\nabla(u-u_h)}_{\oma}+\etaosca$.
Therefore,
\[
\norm{\sh+\bxi+\nabla u_h}^2 \lesssim \sum_{\ver\in\calV} \{ \norm{\nabla(u-u_h)}_{\oma}^2 + [\etaosca]^2 \} \lesssim \norm{\nabla(u-u_h)}^2 + \sum_{\ver\in\calV}[\etaosca]^2,
\]
which completes the proof of global efficiency~\eqref{eq:global_efficiency}.

\section{Proof of Theorem~\ref{thm:flux_stability}}\label{sec:flux_thm_proof}

We will show here that Theorem~\ref{thm:flux_stability} follows easily from
Theorem~\ref{thm:aposteriori}, while using the finite element solution $u_h$
of~\eqref{eq:num_scheme} as an auxiliary ingredient of the proof. Here
it is enough to take $p^\prime=1$ for the polynomial degree
in~\eqref{eq:num_scheme}. Henceforth, suppose that $f \in \calP_{p-1}(\calT)$
and that $\bxi \in \RTNl$ and construct $\sh$ by~\eqref{eq:flux} using the
local minimization problems~\eqref{eq:sha_def}.

Since $\sh \in \Hdivo \cap \RTN$ satisfies $\nabla{\cdot}\sh =  \Pih f
= f$ and $\sh {\cdot} \bn=0$ on $\GN$ by Lemma~\ref{lem:fl_rec}, we have
\begin{equation}\label{eq:proof_thm_flux_1}
\min_{\substack{\bvh  \in \Hdivo\cap \RTN \\ \nabla{\cdot}\bvh = f\text{ in }\Om \\ \bvh {\cdot} \bn = 0\text{ on }\GN }} \norm{\bvh + \bxi } \leq \norm{\sh + \bxi}.
\end{equation}
Therefore it remains only to show that $\norm{\sh + \bxi}$ is bounded by the
right-hand side of~\eqref{eq:stability_rtn}. First, for each vertex $\ver \in
\calV$, note that $\psia f \in \calP_p(\Ta)$ and that $\psia \bxi \in
\RTNa$. Thus the oscillation terms $\etaosca$ given by~\eqref{eq:osc}
vanish. Consequently, the global efficiency
bound~\eqref{eq:global_efficiency} implies that $\norm{\sh + \bxi + \nabla
u_h} \lesssim \norm{\nabla (u-u_h)}$. Furthermore, the stability of the
Galerkin method (recall that the finite element solution $u_h$
of~\eqref{eq:num_scheme} is an orthogonal projection of the weak solution $u$
of~\eqref{eq:weak_pde} from $H^1_*(\Om)$ to $\Vh \subset
H^1_*(\Om)$) implies that $\norm{\nabla u_h}\leq \norm{\nabla u}$.
Therefore we deduce that
\begin{equation}\label{eq:proof_thm_flux_2}
\norm{\sh + \bxi} \leq \norm{\sh + \bxi + \nabla u_h} + \norm{\nabla u_h} \lesssim \norm{\nabla u}.
\end{equation}
Since $u$ is the weak solution of~\eqref{eq:weak_pde}, the equivalence of
primal and dual formulation of elliptic problems implies that
\begin{equation}\label{eq:proof_thm_flux_3}
\norm{\nabla u} = \max_{v\in H^1_*(\Om)\setminus\{0\}}\frac{ (f,v) - (\bxi , \nabla v)}{\norm{\nabla v}} =  \min_{\substack{\bv \in \Hdivo \\ \nabla{\cdot}\bv = f \text{ in }\Om \\ \bv{\cdot} \bn = 0 \text{ on }\GN}} \norm{\bv + \bxi} .
\end{equation}
The combination of the bounds~\eqref{eq:proof_thm_flux_1},~\eqref{eq:proof_thm_flux_2}, and~\eqref{eq:proof_thm_flux_3} yields~\eqref{eq:stability_rtn}.

\section{Proof of Theorem~\ref{thm:flux_stability_2}}\label{sec:flux_stability_2}
As in Section~\ref{sec:flux_thm_proof}, we merely employ the weak
solution $u$ of~\eqref{eq:weak_pde} and its finite element approximation
$u_h$ of~\eqref{eq:num_scheme} as tools. For this purpose, we now
set $\GD = \DO$ and $\GN = \emptyset$, so that $H^1_*(\Om) =
H^1_0(\Om)$, and we choose the auxiliary polynomial degree $p^\prime=1$. We
will construct an equilibrated flux $\shd$ in the discrete minimization set
of the left-hand side of~\eqref{eq:flux_stability_min} such that $\norm{\shd
+ \psid \bxi}$ is bounded by the right-hand side
of~\eqref{eq:flux_stability_2}. The key idea is to write
\begin{equation}\label{eq:shd_definition}
\begin{aligned}
\shd \coloneqq  \shu + \shc , & & &  \shu \coloneqq \sum_{\ver\in\calV} w_{\ver} \shad,
\end{aligned}
\end{equation}
where $\shu$ is an uncorrected high-order flux obtained from the local fluxes
$\shad$, for all $\ver\in\calV$ (see Subsection~\ref{sec:shu_construction}),
and where $\shc$ is a global, low-order, correction flux (see
Subsection~\ref{sec:shc_construction}), and the weights $w_\ver$ result from
\begin{equation}\label{eq:psid_weight_def}
\psid = \sum_{\ver\in\calV} w_{\ver} \psia \quad\text{in }\Om, \qquad w_{\ver}=\psid(\ver), \quad \forall\, \ver\in\calV.
\end{equation}
The correction term $\shc$ is needed, since it will be found below that
$\nabla{\cdot}\shu = \psid f - \nabla \psid {\cdot}(\bxi +\nabla u_h)$; thus
we shall build $\shc$ (by posing a global low-order minimization problem) so
that it satisfies $\nabla{\cdot} \shc = \nabla \psid {\cdot}\nabla u_h$, in
order to ensure that $\shd$ satisfies the divergence constraint required by
the discrete minimization set of the left-hand side of
\eqref{eq:flux_stability_min} (see subsection~\ref{sec:shd_construction}).
The stability properties of $\shc$ then rely on
Theorem~\ref{thm:flux_stability} (in the low-order case $p'=1$), whereas the
stability properties of $\sha$ are established by using similar ideas to
those that were used to prove Theorem~\ref{thm:aposteriori}.

\subsection{Construction of $\shu$}\label{sec:shu_construction}

We construct locally a higher-order $\shad \in \Hdiva\cap \RTNa$ for each
$\ver\in \calV$, similarly to the construction of $\sha$ from
Section~\ref{sec:equilibrated_flux}. We do so in the context $\GD = \DO$ and
$\GN = \emptyset$; consequently, the sets of vertices simplify to interior
and boundary ones $\calVint = \{ \ver \in \calV; \; \ver \in \Om\}$ and
$\calVext = \{ \ver \in \calV; \; \ver \in \DO\}$. For each $\ver\in \calV$,
let the subset $\Gamma_{\ver}\subset \oma$ be given by those faces where the
hat function $\psia$ vanishes, i.e.~$\Gamma_{\ver}\coloneqq \{x \in \p\oma,\;
\psia(x) =0 \}$. Equivalently, $\Gamma_{\ver}$ is composed of all faces $F\in
\calF$ that are contained in $\p\oma$ and that do not contain the vertex
$\ver$; we denote this corresponding set of faces by $\Fa$. If
$\ver\in\calVint$ is an interior vertex, then $\Gamma_{\ver} = \p\oma$. If
$\ver\in\calVext$, then any interior face $F \in \calF\setminus \calFext$
such that $F \subset \p\oma $ necessarily belongs to $\Fa$. If $\Fa$ only
consists of such interior faces, then $\Gamma_{\ver} = \p\oma \setminus \GD =
\p\oma\setminus \DO$. However, $\Gamma_{\ver}$ and $\p\oma \setminus \GD$ do
not coincide for a vertex $\ver$ where $\Gamma_{\ver}$ includes boundary
faces $F\subset \GD$ opposite to the vertex $\ver$. In any case,
$\Gamma_{\ver}$ is always a strict subset of $\p\oma$ for boundary vertices.

For each $\ver\in \calV$, we let the discrete space $\Vad$ be defined by
\[
\Vad  \coloneqq  \left\{\bvh \in \Hdiva\cap \RTNa ,\quad \bvh{\cdot} \bn =0\text{ on }\Gamma_{\ver} \right\} \qquad \forall\, \ver\in\calV.
\]
For interior vertices, $\Vad$ coincides with the space $\Va$ of Section~\ref{sec:equilibrated_flux}. However, in general $\Vad \neq \Va$ for boundary vertices.
In this section, $\Vad$ is used instead of $\Va$ for a technical point appearing in the proof of Lemma~\ref{lem:shd_equilibrium} below concerning~\eqref{eq:shd_equilibrium_2} in situations where some elements may have all faces belonging to the boundary.
For each $\ver\in \calV$, analogously to~\eqref{eq:sha_def}, we let $\shad  \in \Vad$ be defined by
\begin{equation}\label{eq:shad_def}
\shad \coloneqq \arg \min_{\substack{\bvh \in \Vad \\ \nabla{\cdot}\bvh  = g_h^{\ver} }} \norm{\bvh + \psia( \bxi + \nabla u_h)}_{\oma},
\end{equation}
where $g_h^{\ver}$ is given by~\eqref{eq:ga}. Note that $\shad$ is
well-defined for all interior vertices $\ver\in \calVint$, since
$(g_h^{\ver}, 1)_\oma =0$ for $\ver\in \calVint$. In the case of boundary
vertices $\ver\in\calVext$, $\shad$ is also well-defined, since there are
always at least some faces of $\p\oma$ that are not in $\Gamma_{\ver}$.
Finally, the extension by zero of $\shad$ from $\oma$ to $\Om$ is again
$\Hdivo$-conforming. The uncorrected high-order flux function $\shu \in
\RTN \cap \Hdivo $ is then defined by
\begin{equation}\label{eq:shu_def}
 \shu \coloneqq \sum_{\ver\in\calV} w_{\ver} \shad.
\end{equation}
We have the following key result.
\begin{lemma}\label{lem:shad_stability}
Let $\shad$ be defined by~\eqref{eq:shad_def}. Then, we have
\begin{equation}\label{eq:shad_stability}
\norm{\shad + \psia (\bxi + \nabla u_h)}_{\oma} \lesssim \norm{\nabla(u-u_h)}_{\oma}.
\end{equation}
\end{lemma}
\begin{proof}
Similarly to the proof of Lemma~\ref{lem:main_stability_bound}, we
apply~\cite[Thm.~7]{BraessPillweinSchoberl2009} in the case of two space
dimensions, and~\cite[Thm~2.2]{ErnVohralik2016a} in the case of three
space dimensions, to deduce that
\[
\norm{\shad + \psia (\bxi + \nabla u_h)}_{\oma} \lesssim  \max_{v\in \widetilde{H}^1_*(\oma)\setminus\{0\}}
\frac{(f,\psia v)_{\oma}-(\bxi+\nabla u_h,\nabla(\psia v))_{\oma}}{\norm{\nabla v}_{\oma}},
\]
where $\widetilde{H}^1_*(\oma)\coloneqq H^1(\oma)/\R$ if $\ver\in\calVint$,
and $\widetilde{H}^1_*(\oma)\coloneqq H^1_{\p\oma\setminus
\Gamma_{\ver}}(\oma)$ if $\ver\in\calVext$. For any $\ver\in\calV$, we have
$\psia v \in H^1_0(\oma)$ for any $v \in \widetilde{H}^1_*(\oma)$, where we
use the fact that $\psia$ vanishes on $\Gamma_{\ver}$ by definition in the
case of $\ver\in\calVext$. We then use~\eqref{eq:weak_pde}, recalling that
$\GD = \DO$ and that $H^1_*(\Om) = H^1_0(\Om)$ here, so that
\[
    (f,\psia v)_{\oma}-(\bxi+\nabla u_h,\nabla(\psia v))_{\oma} = (\nabla (u - u_h),\nabla(\psia v))_{\oma}.
\]
Thus we obtain~\eqref{eq:shad_stability} from an  application of the Cauchy--Scwharz inequality and the Poincar\'e inequality, using $\norm{\nabla(\psia v)}_{\oma} \lesssim\norm{\nabla v}_{\oma}$ for $v \in \widetilde{H}^1_*(\oma)$, as in Section~\ref{sec:proof_loc_eff}.
\end{proof}

\subsection{Construction of $\shc$}\label{sec:shc_construction}

We now select a global low-order correction $\shc \in \Hdivo \cap
\RTNc$ such that $\nabla {\cdot}\shc = \nabla \psid {\cdot} \nabla u_h$ in
$\Om$, $\shc {\cdot} \bn =0$ on $\Gd$, and such that
\begin{equation}\label{eq:shc}
\norm{\shc } \lesssim \max_{v \in \Hdag \setminus\{0\}} \frac{(\nabla \psid {\cdot} \nabla u_h, v)}{\norm{\nabla v}}.
\end{equation}
This is possible by applying Theorem~\ref{thm:flux_stability}.
Indeed, we employ it in the setting where the Neumann part of the boundary $\DO$ is given by $\Gd$, where the Dirichlet part is $\DO\setminus \Gd$, and where the scalar datum is given by $\nabla \psid {\cdot} \nabla u_h \in \calP_{p^\prime-1}(\calT)$ and the vector datum is zero.
The data compatibility condition for the case where $\Gd = \DO$ is guaranteed, since we can then choose $\psid$ as the test function in~\eqref{eq:num_scheme} (recall that $\psid\in\Vh$ if $\Gd = \DO$), thereby yielding
\[
(\nabla \psid {\cdot} \nabla u_h,1) = (\nabla u_h,\nabla \psid) = (f,\psid)-(\bxi,\nabla\psid)=0,
\]
where the last identity is obtained from the hypothesis on $f$ and $\bxi$ of Theorem~\ref{thm:flux_stability_2} for the case $\Gd=\DO$.
It follows from~\eqref{eq:shc} that
\begin{equation}\label{eq:shc_stability}
\norm{\shc } \leq \Comdag h_{\Om} \norm{\nabla \psid}_{\infty} \norm{\nabla u_h},
\end{equation}
where $\Comdag$ is the constant of the Poincar\'e inequality for $\Hdag$ and
$h_{\Omega}$ is the diameter of~$\Omega$.

\subsection{Admissibility of $\shd$} \label{sec:shd_construction}

Recalling the definitions of $\shu$ from Subsection~\ref{sec:shu_construction} and of $\shc$ from Subsection~\ref{sec:shc_construction}, we define $\shd \in \Hdivo \cap \RTN$ by~\eqref{eq:shd_definition}.
We now check that $\shd$ belongs to the minimization set of the left-hand side of~\eqref{eq:flux_stability_2}.
\begin{lemma}\label{lem:shd_equilibrium}
Let $\shd$ be defined by~\eqref{eq:shd_definition}. Then
\begin{subequations}
\begin{alignat}{2}
\nabla {\cdot} \shd  & = \psid f - \nabla \psid {\cdot} \bxi  & & \quad\text{in }\Om, \label{eq:shd_equilibrium_1} \\
\shd{\cdot}\bn  & = 0  & & \quad \text{on } \Gd.\label{eq:shd_equilibrium_2}
\end{alignat}
\end{subequations}
\end{lemma}
\begin{proof}
Since $f\in \calP_{p-1}(\calT)$ and since $\bxi \in \RTNl \subset \calP_{p}(\calT;\R^{\dim})$, we find that $\nabla {\cdot} \shad = \psia f - \nabla \psia {\cdot} ( \bxi + \nabla u_h)$ from~\eqref{eq:shad_def} and~\eqref{eq:ga}.
In consequence of \eqref{eq:psid_weight_def} and of the definition of $\shu$ in \eqref{eq:shu_def}, by proceeding as in the proof of Lemma~\ref{lem:fl_rec}, we then obtain $\nabla{\cdot}\shu = \psid f - \nabla \psid {\cdot}( \bxi + \nabla u_h)$ in $\Om$.
Therefore, using the definition of the flux correction $\shc$, it is found that
\[
\nabla {\cdot} \shd = \nabla{\cdot} \shu + \nabla{\cdot}\shc = \psid f - \nabla\psid {\cdot} \bxi,
\]
thereby showing~\eqref{eq:shd_equilibrium_1}.

Next, we show~\eqref{eq:shd_equilibrium_2}. Recall that $\shc {\cdot} \bn=0$ on $\Gd$ by construction, and that $\Gd = \cup_{F\in \calFdag}$, where $\calFdag$ is the set of faces $F\in \calFext$ such that $\psid|_F = 0$.
Consider a face $F \in \calFdag$, with its corresponding element $K \in \calT$. Then, for any of the $d$ vertices $\ver\in F $, we have $w_{\ver} = \psid(\ver) =0 $ by definition. Therefore $w_{\ver}\shad {\cdot} \bn =0$ on $F$ trivially for all $\ver \in F$. Since $K$ is a simplex, there is a unique remaining vertex $\ver$ opposing the face $F$. Thus $F\subset\Gamma_{\ver}$ and consequently $\shad {\cdot} \bn =0$ on $F$ by the definition of the space $\Vad$ and of the set $\Gamma_{\ver} = \{x \in \p\oma, \; \psia(x) =0\}$. In summary, for any $F\in \calFdag$, we have $w_{\ver} \shad {\cdot} \bn =0$ on $F$ for all $\ver \in \VK$, where $F\subset \overline{K}$, whence the assertion~\eqref{eq:shd_equilibrium_2} follows. 
\end{proof}

\subsection{Proof of~\eqref{eq:flux_stability_2}}
Recall that we consider the solution $u$ of~\eqref{eq:weak_pde} and
$u_h$ of~\eqref{eq:num_scheme} in the context $\GD = \DO$ and $\GN =
\emptyset$, so that $H^1_*(\Om) = H^1_0(\Om)$. Thus, we obtain the bound
\begin{equation}\label{eq:stability_2_u_value}
\norm{\nabla u_h} \leq \norm{\nabla u} = \max_{v \in H^1_0(\Om)\setminus\{0\}} \frac{(f,v)-(\bxi,\nabla v)}{\norm{\nabla v}}.
\end{equation}
Therefore, our last goal is to show that $\norm{\shd + \psid \bxi}$ can be bounded in terms of $\norm{\nabla u_h}$ and $\norm{\nabla u}$.

For each $K\in \calT$, we have $\shd|_K = \shc|_K + \sum_{\ver\in\VK} (w_{\ver}\shad)|_K$ from~\eqref{eq:shd_definition}.
So, the triangle inequality and the Cauchy--Schwarz inequality imply that
\begin{equation}\label{eq:shd_bound_1}
\norm{\shd + \psid \bxi }^2  \lesssim \norm{\shc}^2 + (\dim+1)\sum_{K\in\calT} \sum_{\ver\in\VK} \abs{w_{\ver}}^2 \norm{ \shad  + \psia (\bxi + \nabla u_h)}_{K}^2 + \norm{\psid\nabla u_h}^2,
\end{equation}
where we have also used~\eqref{eq:psid_weight_def} to obtain $\psid (\bxi + \nabla u_h)|_K = \sum_{\ver\in\VK} [w_{\ver} \psia (\bxi+\nabla u_h)]|_K$.
Lemma~\ref{lem:shad_stability} and a counting argument show that
\begin{equation}\label{eq:shd_bound_2}
\begin{split}
\sum_{K\in\calT}\sum_{\ver\in\VK} \abs{w_{\ver}}^2 \norm{\shad+\psia (\bxi + \nabla u_h)}_{K}^2 &\leq \norm{\psid}_{\infty}^2 \sum_{\ver\in\calV}\norm{\shad+\psia (\bxi + \nabla u_h)}_{\oma}^2 \\
& \lesssim \norm{\psid}_{\infty}^2 \sum_{\ver\in\calV}\norm{\nabla(u-u_h)}_{\oma}^2 \\
& \lesssim \norm{\psid}_{\infty}^2 \norm{\nabla(u-u_h)}^2,
\end{split}
\end{equation}
where the first inequality uses the fact that $\psid \in H^1(\Om)\cap \calP_1(\calT)$ and that the coefficients $\{w_{\ver}\}_{\ver\in\calV}$ are the nodal values of $\psid$ at the vertices of the mesh.
Therefore, the combination of~\eqref{eq:shd_bound_1} with~\eqref{eq:shc_stability} and
\eqref{eq:shd_bound_2} yields
\[
\norm{\shd + \psid \bxi}
\lesssim \Comdag h_{\Omega} \norm{\nabla \psid}_{\infty}\norm{\nabla u_h} + \norm{\psid}_{\infty}(\norm{\nabla(u-u_h)}+\norm{\nabla u_h}).
\]
We finally deduce~\eqref{eq:flux_stability_2} from the above inequality and from \eqref{eq:stability_2_u_value}.

\appendix
\normalsize

\section{Application to a posteriori error estimates on meshes with hanging nodes}\label{sec:appendix}

Equilibrated flux a posteriori error estimates for meshes with hanging nodes
are developed in~\cite{DolejsiErnVohralik2016} where the equilibration is
performed on patches $\om_{\verbis}$ corresponding to the support of hat
functions $\psi_{\verbis}$ associated with non-hanging nodes $\verbis$ of the
computational mesh and forming a partition of unity of the computational
domain, see~\cite[Assumption~2.1]{DolejsiErnVohralik2016}. It turns out that
a slight extension of the equilibration patch $\om_{\verbis}$
of~\cite{DolejsiErnVohralik2016} enables the removal of the usual dependence
of a posteriori efficiency constants on the number of levels of hanging
nodes, thereby allowing for a completely arbitrary number of levels of
hanging nodes. More precisely, it suffices to extend the equilibration patch
$\om_{\verbis}$ so that all the products $h_{\om_{\verbis}} \norm{\nabla
\psi_{\verbis}}_{\infty, \om_{\verbis}}$ are uniformly bounded. Then,
applying Theorem~\ref{thm:flux_stability_2} where the patch $\om_{\verbis}$
of~\cite{DolejsiErnVohralik2016} corresponds here to the domain $\Om$, and
the hat function $\psi_{\verbis }$ of~\cite{DolejsiErnVohralik2016}
corresponds here to the function $\psid$, we infer that the factor
$h_{\om_{\verbis}} \max_{\hat {\verbis } \in \widehat{\calV}_{\verbis}}
\norm{\nabla \psi_{\hat {\verbis}}}_{\infty,\om_{\hat {\verbis}}}$
of~\cite[Theorem~3.12]{DolejsiErnVohralik2016} can be replaced by the factor
$h_{\om_{\verbis}} \norm{\nabla \psi_{\verbis}}_{\infty, \om_{\verbis}}$,
i.e. $h_{\Omega}\norm{\nabla \psid}_{\infty}$ in the present notation. The
extension of the equilibration patch is illustrated in
Figure~\ref{fig_ext_equil}. This extension typically entails including
several layers of fine elements, so as to ensure that the factors
$h_{\om_{\verbis}} \norm{\nabla \psi_{\verbis}}_{\infty, \om_{\verbis}}$ are
uniformly bounded.
The price to pay to achieve robustness with respect to the level of hanging nodes is thus a somewhat more expensive computation of the equilibrated flux.
The proof of Theorem~\ref{thm:flux_stability_2} in Section~\ref{sec:flux_stability_2} shows that this cost can be significantly reduced to the solution of two low-order systems over the extended patch, followed by local high-order corrections within the extended patch.

\begin{figure}
\centerline{\includegraphics[width=0.3\textwidth]{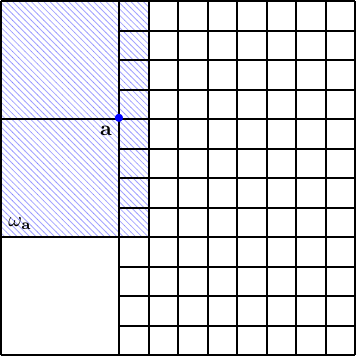} \quad \includegraphics[width=0.3\textwidth]{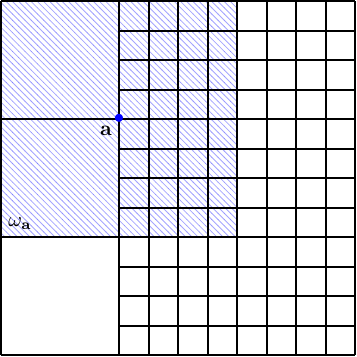}}
\label{fig_ext_equil}
\caption{Original equilibration patches of reference~\cite{DolejsiErnVohralik2016} (left) and extended equilibration patches necessary for estimates robust with respect to an arbitrary number of levels of hanging nodes (right)}
\end{figure}

\end{document}